\documentclass[english]{amsart}
\usepackage[T1]{fontenc}
\usepackage[latin9]{inputenc}
\usepackage{amsthm}
\usepackage{amstext}
\usepackage{amssymb}

\makeatletter
\numberwithin{equation}{section} 
\numberwithin{figure}{section} 
\theoremstyle{plain}
\theoremstyle{plain}
\newtheorem{thm}{Theorem}
  \theoremstyle{plain}
  \newtheorem{prop}[thm]{Proposition}
  \theoremstyle{plain}
  \newtheorem{lem}[thm]{Lemma}

 \usepackage{amsfonts}\usepackage{array}\usepackage{delarray}\usepackage{amsthm}\usepackage{graphics}

 \usepackage{epsfig} 
 \usepackage{graphics} 

\def\C{\mathbb C}
\newcommand{\abs}[1]{| #1 |}
\newcommand{\Abs}[1]{\left| #1\right|}
\def\bcases{\begin{cases}}
\def\bar{\overline}
\newcommand{\br}[1]{\left(#1\right)}
\newcommand{\bk}[1]{\left[#1\right]}
\def\ecases{\end{cases}}
\newcommand{\cl}{\overline}
\def\dbar{\overline\partial}
\newcommand{\dd}{\delta}

\newcommand{\CC}{\mathbf{C}}

\newcommand{\e}{\epsilon}
\newcommand{\Frac}{\displaystyle\frac}
\newcommand{\im}{\text{\rm Im}\,}

\newcommand{\norm}[1]{\left\| #1\right\|}

\newcommand{\p}{\partial}

\newcommand{\R}{\mathbb R}
\newcommand{\re}{\text{\rm Re}\,}
\newcommand{\set}[1]{\left\{ #1\right\}}
\def\sm{\setminus}

\newcommand{\To}{\longrightarrow}

\newcommand{\W}{\Omega}
\def\w{\omega}
\newcommand{\z}{\zeta}

\newtheorem{main theorem}{Main Theorem}

\newtheorem*{problem}{Problem}
\newtheorem{problem 1}{Problem 1}
\newtheorem{problem 2}{Problem 2}
\newtheorem{problem 3}{Problem 3}
\theoremstyle{definition}

\newtheorem{example}{Example}

\newcommand{\bea}{\begin{eqnarray*}}
\newcommand{\eea}{\end{eqnarray*}}

\newcommand{\be}{\begin{equation}}
\newcommand{\ee}{\end{equation}}
\newcommand{\bmult}{\begin{multline}}
\newcommand{\emult}{\end{multline}}

\title[Supnorm Estimate]
{On Supnorm Estimates for $\cl\partial$ on infinite type convex domains in $\C^2$}

\author{John Erik Forn\ae ss, Lina Lee, Yuan Zhang }

\makeatother

\usepackage{babel}

\makeatother

\usepackage{babel}

\begin{document}
\begin{abstract}
In this paper, we study the $\overline{\partial}$ equation on some
convex domains of infinite type in $\C^{2}$. In detail, we prove
that sup norm estimates hold for infinite exponential type domains
provided the exponent is less than 1. 
\end{abstract}
\maketitle

\section{Introduction}

Let $\W$ be a smooth bounded domain in $\C^{n}$. Given a smooth
$\bar{\p}$ closed $(0,1)$-form $f$, one fundamental question is
to study the supnorm estimates for the solutions of $\bar{\p}u=f$. A positive answer
 is well known when $\W$ is strictly convex or strongly pseudoconvex,
by constructing integral formulas for $\bar{\p}$ equations. Supnorm
and Hölder estimates have been established by Grauert and Lieb \cite{GL},
Henkin \cite{henkin} and Kerzman \cite{Kerzman} etc. Indeed, supnorm
and Hölder estimates with order $\frac{1}{m}$ still hold if $\W\subset\C^{n}$
is convex of finite type with type $m$ (see \cite{DF} \cite{DFF}).
However, for infinite type convex domains, even supnorm estimates
are still
 unknown. We discuss in this paper some examples of convex
domains of infinite type in $\C^{2}$. In particular, for the bidisk
rounded off by infinite exponential type hypersurface, we prove if
the exponent is less than $1$, the solutions given by the integral representation
to the $\bar{\p}$ equations have supnorm estimates. However, whether the exponent $1$ is optimal for the supernorm estimate is
still unclear to us.

\subsection{Infinite Type convex domains in $\C^{2}$}

\begin{problem} Assume that $\W$ is a bounded convex domain in $\C^{2}.$
Can one solve the $\overline{\partial}-$equation with supnorm estimates
on $\W?$ More precisely, does there exist a constant $C=C(\W)$ so
that whenever $f=\sum_{i=1}^{2}f_{i}(z)d\overline{z}_{i}$ on $\W$
for bounded measurable functions $f_{i}$ on $\W$, and $\overline{\partial}f=0$
in the sense of distributions, then is
there a measurable function
$u(z)$ on $\W$ so that $\overline{\partial}u=f$ and $\|u\|_{\infty}\leq C\|f\|_{\infty}.$
\end{problem}

We note that we put no restrictions on the smoothness of the boundary
of $\W.$ We note that if we moreover would like to prove that the
constant $C(\W)$ only depends on the diameter of $\W$, which is
the case in the corresponding problem for $L^{2}$ estimates \cite{Hormander},
then it suffices to study the case when the domain has smooth boundary
and is strongly convex.

There is a natural integral kernel to solve the $\overline{\partial}$
problem, namely the Henkin Kernel \cite{henkin}:

\begin{multline}
4\pi^{2}u(z)=\int_{\z\in\p\W}\frac{\rho_{\z_{1}}(\cl\z_{2}-\cl z_{2})-\rho_{\z_{2}}(\cl\z_{1}-\cl z_{1})}{\bk{\rho_{\z_{1}}(z_{1}-\z_{1})+\rho_{\z_{2}}(z_{2}-\z_{2})}\abs{\z-z}^{2}}f\wedge\w(\z)\\
+\int_{\z\in\W}\frac{f_{1}(\cl\z_{1}-\cl z_{1})+f_{2}(\cl\z_{2}-\cl z_{2})}{\abs{\z-z}^{4}}\w(\cl\z)\wedge\w(\z)=Hf+Kf\label{eq:Henkinformula}\end{multline}
 where $\rho$ is a $C^{1}$ defining function of $\W$ and $\w(\z)=d\z_{1}\wedge d\z_{2}$.

For smooth finite type convex domains in $\C^{2}$, the result has
been known for some time (see \cite{Range}). In that paper, Range
used Skoda's technique to construct a different Cauchy-Fantappié kernel
and carried out the
Hölder estimates for $\dbar$ equations.

Henkin in \cite{henkin2} proved the theorem in the case of the bidisc.
In this case, the integral formula is still valid although the domain
does not have smooth boundary.

The methods in the two cases are different. In the finite type case
one can directly estimate the kernel H and show it is uniformly in
$L^{1}$ for points $z\in\W.$ In other words one can estimate the
integral of the absolute value. For the bidisc case, this approach
fails. Instead Henkin uses an argument involving integration by parts
which essentially takes the integral over the flat face $\{|\z_{1}|=1\}$
to an integral over $\{|\z_{2}|=1\}$ and vice versa. For these new
integrals one can integrate the absolute values of the integrands
to prove uniform $L^{1}$ estimates. This idea has been carried over
to more general polyhedra, see \cite{henkin2} even in higher dimension.
It is not known whether one can solve for infinite type smooth convex
domains. 

In this paper we will investigate the case of some convex domains
which have a relatively open part of the boundary which is Levi flat
and another part which is strongly convex. For such domains one can
split the boundary into three pieces. One which is the flat points,
where one can use the method from Henkins bidisc result, one which
is a compact uniformly convex part and a third one which contains
strongly convex points near the flat part. The main difficulty is
to deal with the third part.

This difficulty is already apparent in the case of smooth domains
which are strongly convex except for one infinitely flat point. We
will discuss this case and show that there is a critical exponent
which decides absolute integrability of the Henkin kernel. This information
motivates how to find a rounded off polydisc where one can solve $\overline{\partial}$
with sup norm estimates.

Our main case is the following:

\begin{example} Let $\chi:\R^{+}\cup\set0\To\R^+$ be a smooth function
such that $\chi''(t)\ge0$ everywhere, $\chi''(t)>0$ for all $t\in(1,1+a)$,
and \[
\chi(t)=\begin{cases}
1, & t\in[0,1]\\
1+\exp\br{-\frac{1}{(t-1)^{\frac{\alpha}{2}}}}, & t\in(1,1+\e)\\
t-\eta, & t\ge1+a\end{cases},\]
 where $a>\e, \eta>0$ are small numbers such that $\chi', \chi'',\chi'''>0$.

Let us define a domain $\W\subset\C^{2}$ as follows: \[
\W=\set{\rho(z_{1},z_{2})=\chi(\abs{z_{1}}^{2})+\abs{z_{2}}^{2}<4}.\]
 \end{example}

We show that we can solve $\overline{\partial}$ with supnorm estimates
on this domain. Another interesting example is the flattened ball:
$\W=\{|z|^{2}+|w|^{2}<4,\re w<1\}.$ In this case it is not known
whether one has supnorm estimates. The problem seems to be the boundary
points $(0,1\pm i\sqrt{3})$ where the complex gradients of the functions
$(|z|^{2}+|w|^{2},\re w)$ are not complex linearly independent.

The paper is organized as follows:

In the next section we recall the proof of Henkins bidics theorem
and the strongly convex case. Then in Section 3 we discuss domains
with totally real flat parts. In Section 4 we discuss the example above. We
remark that as an application of this we get supnorm estimates for
some rounded bidisc.

\section{Background Results}

\subsection{Henkin's integral formula on a bidisc}

\begin{prop}\label{henkin2} Let $\W=\set{\abs{z_{1}}<1,\abs{z_{2}}<1}$.
There is a constant $C>0$, so that for any
 $f=f_{1}d\cl z_{1}+f_{2}d\cl z_{2}$ and $\bar{\p}f=0$ on
$\W$, $\|f\|_\infty < \infty$.

 Then \begin{multline}\label{bidisk}
u(z)=\frac{3}{4\pi^{2}}\int_{\W}\frac{f_{1}(\cl\z_{1}-\cl z_{1})+f_{2}(\cl\z_{2}-\cl z_{2})}{\abs{\z-z}^{4}}d\cl\z_{1}\wedge d\cl\z_{2}\wedge d\z_{1}\wedge d\z_{2}\\
+\frac{1}{4\pi^{2}}\int_{\abs{\z_{2}}=1,\abs{\z_{1}}<1}\frac{f_{1}}{\z_{1}-z_{1}}\frac{\cl\z_{2}-\cl z_{2}}{\abs{\z-z}^{2}}d\cl\z_{1}\wedge d\z_{1}\wedge d\z_{2}+\frac{i}{2\pi}\int_{\abs{\z_{2}}<1}\frac{f_{2}(z_{1},\z_{2})}{\z_{2}-z_{2}}d\cl\z_{2}\wedge d\z_{2}\\
-\frac{1}{4\pi^{2}}\int_{\abs{\z_{1}}=1,\abs{\z_{2}}<1}\frac{f_{2}}{\z_{2}-z_{2}}\frac{\cl\z_{1}-\cl z_{1}}{\abs{\z-z}^{2}}d\cl\z_{2}\wedge d\z_{1}\wedge d\z_{2}+\frac{i}{2\pi}\int_{\abs{\z_{1}}<1}\frac{f_{1}(\z_{1},z_{2})}{\z_{1}-z_{1}}d\cl\z_{1}\wedge d\z_{1}\end{multline}
 gives a solution for $\bar{\p}u=f$ on $\W,$ $\|u\|_\infty \leq C \|f\|_\infty.$\\
 \end{prop}

 From now on, we also denote $f\leq C g$ for some constant $C$ by $f \lesssim g $, for simplification. In \cite{henkin2}, the details of the proof was not
given. For the completeness of our result, we give the detailed computation
here. At the end of the section, we will also show the supnorm estimate
of the $\bar{\p}$ equation.

\begin{proof} The usual Henkin's integral formula on a domain $\W=\set{\rho<0}\subset\C^{2}$
is as follows:

\begin{multline}
4\pi^{2}u(z)=\int_{\p\W}\frac{\rho_{\z_{1}}(\cl\z_{2}-\cl z_{2})-\rho_{\z_{2}}(\cl\z_{1}-\cl z_{1})}{\bk{\rho_{\z_{1}}(z_{1}-\z_{1})+\rho_{\z_{2}}(z_{2}-\z_{2})}\abs{\z-z}^{2}}f\wedge d\z_{1}\wedge d\z_{2}\\
+\int_{\W}\frac{f_{1}(\cl\z_{1}-\cl z_{1})+f_{2}(\cl\z_{2}-\cl z_{2})}{\abs{\z-z}^{4}}d\cl\z_{1}\wedge d\cl\z_{2}\wedge d\z_{1}\wedge d\z_{2}\end{multline}

Applying this formula on bidisk, we get \[
4\pi^{2}u(z)=\int_{\W}\cdots+\int_{\p\W}\cdots=\int_{\W}\cdots+\int_{\abs{\z_{1}}=1,\abs{\z_{2}}<1}\cdots+\int_{\abs{\z_{2}}=1,\abs{\z_{1}}<1}\cdots=A_{1}+A_{2}+A_{3},\]
 where \begin{gather*}
A_{1}=\int_{\W}\frac{f_{1}(\cl\z_{1}-\cl z_{1})+f_{2}(\cl\z_{2}-\cl z_{2})}{\abs{\z-z}^{4}}d\cl\z_{1}\wedge d\cl\z_{2}\wedge d\z_{1}\wedge d\z_{2},\\
A_{2}=\int_{\abs{\z_{1}}=1,\abs{\z_{2}}<1}\frac{\cl\z_{1}(\cl\z_{2}-\cl z_{2})}{\cl\z_{1}(z_{1}-\z_{1})\abs{\z-z}^{2}}f\wedge\w(\z)=\int_{\abs{\z_{1}}=1,\abs{\z_{2}}<1}\frac{(\cl\z_{2}-\cl z_{2})}{(z_{1}-\z_{1})\abs{\z-z}^{2}}f_{2}d\cl\z_{2}\wedge\w(\z),\\
A_{3}=\int_{\abs{\z_{2}}=1,\abs{\z_{1}}<1}\frac{\cl\z_{2}(\cl\z_{1}-\cl z_{1})}{\cl\z_{2}(z_{2}-\z_{2})\abs{\z-z}^{2}}f\wedge\w(\z)=\int_{\abs{\z_{2}}=1,\abs{\z_{1}}<1}\frac{(\cl\z_{1}-\cl z_{1})}{(z_{2}-\z_{2})\abs{\z-z}^{2}}f_{1}d\cl\z_{1}\wedge\w(\z)\end{gather*}
 and $\w(\z)=d\z_{1}\wedge d\z_{2}$.

Now let us look at $A_{2}$. We may pick a small $\e>0$ such that
$B(z_{1},\e)=\set{\z:\abs{\z-z_{1}}<\e}\subset D_{1}=\set{\z_{1}:\abs{\z_{1}}<1}$.
Then we apply Stokes' Theorem on $D_{1}\sm B(z_{1},\e)$ for the boundary
integral on $\abs{\z_{1}}=1$: \[
\begin{split}A_{2}= & \int_{\abs{\z_{1}}=1,\abs{\z_{2}}<1}\frac{(\cl\z_{2}-\cl z_{2})}{(z_{1}-\z_{1})\abs{\z-z}^{2}}f_{2}d\cl\z_{2}\wedge\w(\z)\\
= & \int_{\abs{\z_{1}-z_{1}}=\e,\abs{\z_{2}}<1}\frac{(\cl\z_{2}-\cl z_{2})}{(z_{1}-\z_{1})\abs{\z-z}^{2}}f_{2}d\cl\z_{2}\wedge\w(\z)+\int_{D_{1}\sm B(z_{1},\e),\abs{\z_{2}}<1}\frac{\p}{\p\cl{\z_{1}}}\br{\cdots}\w(\cl\z)\wedge\w(\z)\\
= & \int_{\abs{\z_{1}-z_{1}}=\e,\abs{\z_{2}}<1}\frac{(\cl\z_{2}-\cl z_{2})}{(z_{1}-\z_{1})\abs{\z-z}^{2}}f_{2}d\cl\z_{2}\wedge\w(\z)\\
 & -\int_{D_{1}\sm B(z_{1},\e),\abs{\z_{2}}<1}\frac{\cl\z_{2}-\cl z_{2}}{(z_{1}-\z_{1})}\cdot\frac{\z_{1}-z_{1}}{\abs{\z-z}^{4}}f_{2}\w(\cl\z)\wedge\w(\z)\\
 & +\int_{D_{1}\sm B(z_{1},\e),\abs{\z_{2}}<1}\frac{\cl\z_{2}-\cl z_{2}}{(z_{1}-\z_{1})}\cdot\frac{1}{\abs{\z-z}^{2}}\frac{\p f_{2}}{\p\cl\z_{1}}\w(\cl\z)\wedge\w(\z)=B_{1}+B_{2}+B_{3}\end{split}
\]
 As $\e\to0$, we have \begin{gather*}
\begin{split}B_{1} & \To2\pi i\int_{\abs{\z_{2}}<1}\frac{f_{2}(z_{1},\z_{2})}{\z_{2}-z_{2}}d\cl\z_{2}\wedge d\z_{2},\\
B_{2} & \To\int_{\W}\frac{f_{2}(\cl\z_{2}-\cl z_{2})}{\abs{\z-z}^{4}}\w(\cl\z)\wedge\w(\z).\end{split}
\end{gather*}
 We use $\dbar f=0$ and integration by parts to get

\begin{align*}
B_{3} & =\int_{D_{1}\sm B(z_{1},\e),\abs{\z_{2}}<1}\frac{\cl\z_{2}-\cl z_{2}}{(z_{1}-\z_{1})}\cdot\frac{1}{\abs{\z-z}^{2}}\frac{\p f_{2}}{\p\cl\z_{1}}\w(\cl\z)\wedge\w(\z)\\
 & =\int_{D_{1}\sm B(z_{1},\e),\abs{\z_{2}}<1}\frac{\cl\z_{2}-\cl z_{2}}{(z_{1}-\z_{1})}\cdot\frac{1}{\abs{\z-z}^{2}}\frac{\p f_{1}}{\p\cl\z_{2}}\w(\cl\z)\wedge\w(\z)\\
 & =\int_{\abs{\z_{2}}=1,D_{1}\sm B(z_{1},\e)}\frac{f_{1}}{\z_{1}-z_{1}}\frac{\cl\z_{2}-\cl z_{2}}{\abs{\z-z}^{2}}d\cl\z_{1}\wedge\w(\z)+\int_{\abs{\z_{2}}<1,D_{1}\sm B(z_{1},\e)}\frac{f_{1}(\cl\z_{1}-\cl z_{1})}{\abs{\z-z}^{4}}\w(\cl\z)\wedge\w(\z)\\
 & \To\int_{\abs{\z_{2}}=1,\abs{\z_{1}}<1}\frac{f_{1}}{\z_{1}-z_{1}}\frac{\cl\z_{2}-\cl z_{2}}{\abs{\z-z}^{2}}d\cl\z_{1}\wedge\w(\z)+\int_{\W}\frac{f_{1}(\cl\z_{1}-\cl z_{1})}{\abs{\z-z}^{4}}\w(\cl\z)\wedge\w(\z).\end{align*}

Therefore we have \begin{multline}
A_{2}=B_{1}+B_{2}+B_{3}\\
=2\pi i\int_{\abs{\z_{2}}<1}\frac{f_{2}(z_{1},\z_{2})}{\z_{2}-z_{2}}d\cl\z_{2}\wedge d\z_{2}+\int_{\W}\frac{f_{2}(\cl\z_{2}-\cl z_{2})}{\abs{\z-z}^{4}}\w(\cl\z)\wedge\w(\z)\\
+\int_{\abs{\z_{2}}=1,\abs{\z_{1}}<1}\frac{f_{1}}{\z_{1}-z_{1}}\frac{\cl\z_{2}-\cl z_{2}}{\abs{\z-z}^{2}}d\cl\z_{1}\wedge\w(\z)+\int_{\W}\frac{f_{1}(\cl\z_{1}-\cl z_{1})}{\abs{\z-z}^{4}}\w(\cl\z)\wedge\w(\z)\\
=A_{1}+\int_{\abs{\z_{2}}=1,\abs{\z_{1}}<1}\frac{f_{1}}{\z_{1}-z_{1}}\frac{\cl\z_{2}-\cl z_{2}}{\abs{\z-z}^{2}}d\cl\z_{1}\wedge\w(\z)+2\pi i\int_{\abs{\z_{2}}<1}\frac{f_{2}(z_{1},\z_{2})}{\z_{2}-z_{2}}d\cl\z_{2}\wedge d\z_{2}\end{multline}

Calculation of $A_{3}$ is similar. Hence we get (\ref{bidisk}) as
in the Proposition.

 We show next that the solution $u$ in (\ref{bidisk})
is bounded if $|f|$ is bounded. It suffices to estimate the second
and fourth expression on the right side. Since they are symmetric,
it suffices to do the second one.

Notice that $
a+b\gtrsim a^{2/3}b^{1/3}$ if $a>0, b>0$.
 Hence we have $|z-\z|^{2}\gtrsim|z_{1}-\z_{1}|^{2/3}|z_{2}-\z_{2}|^{4/3}$.
Therefore \begin{align*}
\frac{1}{|\z_{1}-z_{1}|}\frac{|\cl\z_{2}-\cl z_{2}|}{|\z-z|^{2}} & \lesssim\frac{1}{|\z_{1}-z_{1}|}\frac{|\cl\z_{2}-\cl z_{2}|}{|\z_{1}-z_{1}|^{2/3}|\z_{2}-z_{2}|^{4/3}}\\
 & =\frac{1}{|\z_1-z_1|^{5/3}}\frac{1}{|\z_2-z_2|^{1/3}}\end{align*}

Then the integrand in the second expression of (\ref{bidisk}) is integrable.\end{proof}


\subsection{Henkin's integral formula on a strongly convex domain}

The estimate of the Henkin's integral formula is well known for strongly
pseudoconvex domains (see \cite{henkin}). Here we include the proof
for strongly convex domains first because the method we use is different
from that in \cite{henkin} and also because the estimtates for other
special domains in section 3 and 4 are based on the estimate we show
below. Hence it would not inconvenience the reader if we skip the
estimate in section 3 and 4 which are the same as in the case for
strongly convex domains.
\begin{prop}
\label{henkinstrconv}Let $\W=\set{\rho<0}\subset\C^{n}$ be a smoothly
bounded strongly convex domain. Let $\dbar u=f$ and $\dbar f=0$
on $\W$, for $f\in C^{\infty}(\W)$. Then we have $\norm u_{\infty}\lesssim\norm f_{\infty}$, where $\norm u_{\infty}=\sup\set{|u(z)|:z\in\cl\W}$. \end{prop}
\begin{proof}
It suffices to show that the integral $Hf$ in (\ref{eq:Henkinformula})
is bounded. Recall the formula for $Hf$: \[
Hf(z)=\int_{\z\in\p\W}\frac{\rho_{\z_{1}}(\cl\z_{2}-\cl z_{2})-\rho_{\z_{2}}(\cl\z_{1}-\cl z_{1})}{\bk{\rho_{\z_{1}}(z_{1}-\z_{1})+\rho_{\z_{2}}(z_{2}-\z_{2})}\abs{\z-z}^{2}}f\wedge\w(\z).\]
 Let $ $$F(z,\z)=\rho_{\z_{1}}(z_{1}-\z_{1})+\rho_{\z_{2}}(z_{2}-\z_{2})$.
Then we have\[
\abs{Hf(z)}\lesssim\int_{\z\in\p\W}\frac{|f_{1}|d\cl{\z_{1}}+|f_{2}|d\cl{\z_{2}}}{(|\im F|+|\re F|)|\z-z|}\wedge\w(\z).\]
 If we let $z=(x_{1}+ix_{2},x_{3}+ix_{4})$ and $\z=(t_{1}+it_{2},t_{3}+it_{4})$,
then \[
\re F(z,\z)=\frac{1}{2}\sum_{j=1}^{4}\frac{\p\rho}{\p t_{j}}(x_{j}-t_{j}).\]
 For fixed $\z\in\p\W$, the Taylor expansion of $\rho$ at $\z$
evaluated at $z\in\cl\W$ is as follows:\begin{align*}
\rho(z) & =2\re F(z,\z)+\frac{1}{2}\sum_{j,k=1}^{4}\frac{\p^{2}\rho}{\p t_{j}\p t_{k}}(x_{j}-t_{j})(x_{k}-t_{k})+O(|z-\z|^{3})\\
 & \ge2\re F(z,\z)+C|z-\z|^{2}+O(|z-\z|^{3}).\end{align*}
 Hence for $\dd>0$ small enough, we have\begin{equation}
|\re F(z,\z)|\ge C'|z-\z|^{2},\quad\text{for }z\in\cl\W,|z-\z|<\dd,\label{eq:ReF}\end{equation}
 since $\re F(z,\z)\le0$. We may assume $\norm{\nabla\rho(\z)}=1$
for all $\z\in\p\W$. Then $|\re F(z,\z)|$ is comparable to the square of the distance
from $z\in\cl\W$ to the real tangent plane to $\p\W$ at $\z\in\p\W$.
Hence if $z\in\cl\W$ and $|z-\z|>\dd$$ $, then $|\re F(z,\z)|\ge\inf\set{|\re F(z,\z)|:|z-\z|=\dd,\, z\in\cl\W}\gtrsim\dd^{2}$.
Therefore we have\[
|Hf(z)|\lesssim\int_{\z\in\p\W,|z-\z|\le\dd}\cdots+\int_{\z\in\p\W,|z-\z|>\dd}\cdots=I+II\]
 and $II\lesssim\norm f_{\infty}$. Therefore it is sufficient to
show $I\lesssim\norm f_{\infty}$. Let $z\in\cl\W$ be fixed such
that $S_{z,\dd}=\set{\z\in\p\W:|z-\z|\le\dd}\neq\emptyset$. May assume
$\dd>0$ is small enough such that $\p\rho/\p t_{3}\approx1$ on $ $$S_{z,\dd}$.
Then we have $d\rho=\rho_{t_{1}}dt_{1}+\rho_{t_{2}}dt_{2}+\rho_{t_{3}}dt_{3}+\rho_{t_{4}}dt_{4}$
and \[
dt_{3}=\frac{-1}{\rho_{t_{3}}}(\rho_{t_{1}}dt_{1}+\rho_{t_{2}}dt_{2}+\rho_{t_{4}}dt_{4}),\quad\text{on }S_{z,\dd}.\]
 Therefore $|d\cl\z_{1}\wedge\w(\z)|\approx dt_{1}dt_{2}dt_{4}$ and
$|d\cl\z_{2}\wedge\w(\z)|\approx dt_{1}dt_{2}dt_{4}$. So we have\begin{equation}
I\lesssim\norm f_{\infty}\int_{(t_{1}-x_1)^{2}+(t_{2}-x_2)^{2}+(t_{4}-x_2)^{2}<\dd}\frac{1}{(|\im F|+|\re F|)|\z-z|}dt_{1}dt_{2}dt_{4}.\label{eq:I-general}\end{equation}
 It is easy to check that $\p\im F/\p t_{4}\neq0$ near $S_{z,\dd}$.
Hence the Jacobian of the coordinate change mapping $\Phi(t)=(t_{1},t_{2},t_{4})\to(t_{1},t_{2},\im F)$
does not vanish on $S_{z,\dd}$. We have\begin{align}
I & \lesssim\norm f_{\infty}\int_{(t_{1}-x_1)^{2}+(t_{2}-x_2)^{2}+(t_{4}-x_2)^{2}<\dd}\frac{1}{(|t_{4}|+|\re F|)|\z_{1}-z_{1}|}dt_{1}dt_{2}dt_{4}\label{eq:I-strconv}\\
 & \lesssim\norm f_{\infty}\int_{(t_{1}-x_1)^{2}+(t_{2}-x_2)^{2}<\dd}\frac{\Big|\ln(|\re F|+\sqrt{\dd})-\ln|\re F|\Big|}{|\z_{1}-z_{1}|}dt_{1}dt_{2}\nonumber \\
 & \lesssim\norm f_{\infty}\int_{(t_{1}-x_1)^{2}+(t_{2}-x_2)^{2}<\dd}\frac{|\ln|\re F||}{|\z_{1}-z_{1}|}dt_{1}dt_{2}\nonumber \end{align}
By (\ref{eq:ReF}), we get\begin{align}
I & \lesssim\norm f_{\infty}\int_{(t_{1}-x_1)^{2}+(t_{2}-x_2)^{2}<\dd}\frac{|\ln|\re F||}{|\z_{1}-z_{1}|}dt_{1}dt_{2}\label{eq:ReFstrconv}\\
 & \lesssim\norm f_{\infty}\int_{(t_{1}-x_1)^{2}+(t_{2}-x_2)^{2}<\dd}\frac{|\ln(|z_{1}-\z_{1}|^{2}+|z_{2}-\z_{2}|^{2})|}{|\z_{1}-z_{1}|}dt_{1}dt_{2}\nonumber \\
 & \lesssim\norm f_{\infty}\int_{(t_{1}-x_1)^{2}+(t_{2}-x_2)^{2}<\dd}\frac{|\ln|z_{1}-\z_{1}||}{|\z_{1}-z_{1}|}dt_{1}dt_{2}.\nonumber \end{align}
Let us use polar cooridnates for $(t_{1},t_{2})$ such that $|z_{1}-\z_{1}|=r$
and $t_{1}=x_{1}+r\cos\theta$. Then we get\[
I\lesssim\norm f_{\infty}\int_{r<\delta}\frac{|\ln r|}{r}r\, dr\lesssim\norm f_{\infty}.\]

\end{proof}

\section{Some convex domains with totally real flat parts}

In this section we shall study smoothly bounded convex domains in
$\C^{2}$ which is strongly convex except
on some totally real flat boundary pieces. We assume those flat parts to be of exponentially infinite type.

As one can see in the proof of Proposition \ref{henkinstrconv}, to
discuss the integrability of the Henkin kernel $H$ in (\ref{eq:Henkinformula}),
it is enough to consider the case when $z$ is close to the boundary
and the integral on a small boundary piece close to $z$. For the domains
we consider in this section, if $z$ is close to the strongly convex boundary
point and if one can choose a small neighborhood around $z$ such
that the boundary piece in that neighborhood is strongly convex, then
we can use the same estimates as in (\ref{eq:I-strconv}) and (\ref{eq:ReFstrconv})
and the integral is finite on that boundary piece. Therefore it is enough to consider the case when $z$ is close to
the exponentially flat points, and evaluate
the integral on $S_{z,\dd}=\set{\z\in\p\W:|z-\z|<\dd}$.

\begin{lem}
\label{lem:boundconvex}If a $C^{2}$-function $\phi:\R^{+}\cup\set0\To\R^{+}\cup\set0$
satisfies $\phi'(t)\ge0$ and $\phi''(t)\ge0$ for all $t\in\R^{+}\cup\set0$,
then $\psi
(x):=\phi(x_{1}^{2}+\cdots+x_{n}^{2})$ is a convex function
on $\R^{n}$. \end{lem}

The proof is straightforward and is omitted here.

\begin{lem}
\label{conv} Let $\phi$ be a smooth function on $\R^{+}\cup\set0$
such that $\phi(0)=\phi'(0)=0$, $\phi''(t)\ge0$, and $\phi'''(t)\ge0$,
for all $t\in\R^{+}\cup\set0$. Let $p,q\in\R^{+}\cup\set0$. Then
we have \[
\phi(q)-\phi(p)-\phi'(p)(q-p)\ge\phi(q-p),\quad\text{if }0\le p\le q,\]
and \[
\phi(q)-\phi(p)-\phi'(p)(q-p)\ge\phi''\br{\frac{p+q}{2}}\br{\frac{p-q}{2}}^{2},\quad\text{if }0\le q<p.\]

\end{lem}

\begin{proof}
Let us assume $0\le p\le q$ and let $q=p+s$. We want to show the
following: \[
g(p,s)=\phi(p+s)-(\phi(p)+s\phi'(p))-\phi(s)\ge0,\quad\forall p,s\ge0.\]
 Obviously $g(0,s)=0.$ Moreover $\p g/\p p=\phi'(p+s)-\phi'(p)-s\phi''(p)\ge0$
since $\phi'$ is a convex function on $\R^{+}\cup\set0$. Therefore
$g(p,s)\ge0$ for all $p,s\ge0$.

Now assume $0\le q<p$ and let $p=q+s$, $s>0$. Let $h(q,s)=\phi(q+s)-\phi'(q+s)s$.
Then we have $h(q,0)=\phi(q)$ and $\p h/\p s=-\phi''(q+s)s\le0$.
Hence we get $h(q,0)-h(q,s)\ge h(q,s/2)-h(q,s)$. Note that $\p^{2}h/\p s^{2}=-\phi'''(q+s)\le0$.
Therefore we have \[
\frac{h\br{q,s/2}-h(q,s)}{-s/2}\le\frac{\p h}{\p s}(q,s/2).\]
Hence\[
\phi(q)-\phi(p)-\phi'(p)(q-p)=h(q,0)-h(q,s)\ge\frac{s^{2}}{4}\phi''\br{q+\frac{s}{2}}=\phi''\br{\frac{p+q}{2}}\br{\frac{p-q}{2}}^{2}.\]\end{proof}


We construct two smoothly bounded convex domains $\W_{1},\W_{2}\subset\C^{2}$
as follows. Locally, we study $\p\W_{1}\cap B(0,\e)\subset\set{\rho(z)=\re z_{2}+\exp(-1/|z_{1}|^{\alpha})=0}$
and $\p\W_{2}\cap B(0,\e)\subset \set{\rho(z)=\re z_{2}+\exp(-1/|\re z_{1}|^{\alpha})=0}$. Here $\epsilon$ is small enough so that $\exp(-1/|z_{1}|^{\alpha})$
and $\exp(-1/|\re z_{1}|^{\alpha})$ are convex if $|z_{1}|<\e$. As one can see, the boundaries of $\W_{1}$ and $\W_{2}$ are strongly
convex everywhere except along the imaginary $z_2$ axis in $\W_1$ and along both imaginary $z_1, z_2$ axis in $\W_2$. Moreover,  in the origin in $\W_1$ and in the imaginary $z_1$ axis in $\W_2$, the boundaries are of exponentially infinite type. To get bounded convex domains in $\CC^2$,  we need the following patching Proposition:

\begin{prop}
\label{pro:roundoff}Let $\W\subset\C^{n}$ be a smoothly bounded
domain and $0\in\p\W$. Then there exists a domain $\tilde \W\subset\subset\C^{n}$
such that $\p{\tilde \W}$ is smooth and $\p\tilde \W\cap B(0,\e)=\p\W\cap B(0,\e)$
for some small $\e>0$ and that $\p\tilde \W$ is strongly convex except possibly
on $\p\tilde \W\cap B(0,2\e)$. Moreover if $\W$ is convex, then $\tilde \W$ can be chosen to be bounded convex. \end{prop}
\begin{proof}
We use Lemma \ref{lem:boundconvex} with the following function:
\begin{equation}\label{bounds}
\psi(t)=\begin{cases}
0, & t\in[0,\e^{2}),\\
e^{-1/(t-\e^{2})}(t^{2}-\e^{4}), & t\ge\e^{2}.\end{cases}\end{equation}
Then $\psi:[0,\infty)\To[0,\infty)$ is a smooth convex function and we have
$\psi'(t)>0,
\psi''(t)>0$ if $t>\e^{2}$. Moreover, there exists $\beta>0$ such that $\psi''(t)>\beta$
if $t>4\e^{2}$. Let $\rho$ be a smooth local defining function of
$\W$ near $0$, i.e., $\W\cap U=\set{\rho<0}\cap U$ for some neighborhood
of $U$. We consider $\tilde{\W}=\set{z\in\C^{n}:\rho+M\psi(|z|^{2})<0}$
with some appropriate constant $M>0$. Since $\psi(t)\To\infty$ as
$t\to\infty$, clearly $\tilde{\W}\subset\subset\C^{n}$. Also $\tilde{\W}\cap B(0,\e)=\W\cap B(0,\e)$,
where $\e>0$ is chosen small enough that $B(0,\e)\subset U$. We
may choose $M$ large enough that $\tilde{\W}$ is strongly convex
on $\p\tilde{\W}\sm B(0,2\e)$. \end{proof}

Using Proposition \ref{pro:roundoff}, we now have constructed two
smoothly bounded convex domains $\W_{1},\W_{2}\subset\C^{2}$ such
that\begin{equation}  \label{domain}
\begin{split}
\W_{1} & =\set{\rho(z)=\re z_{2}+\exp(-1/|z_{1}|^{\alpha})+M\psi(|z|^{2})<0},\quad\text{and }\\
\W_{2} & =\set{\rho(z)=\re z_{2}+\exp(-1/|\re z_{1}|^{\alpha})+M\psi(|z|^{2})<0},\end{split} \end{equation}
where $\psi$ is given as in (\ref{bounds}).
\begin{thm}
\label{thm:convexinf}Let $\alpha<1$ on both of the two cases in (\ref{domain}), then the solution to the $\dbar$-equation $\dbar u=f$,
for $f=f_{1}d\cl z_{1}+f_{2}d\cl z_{2}$, $f\in C^{\infty}(\cl\W_{j})$
and $\dbar f=0$, satisfies $\norm u_{\infty}\lesssim\norm f_{\infty}$. \end{thm}
\begin{proof}
We show that the Henkin integral is bounded on $\W_{1}$ and $\W_{2}$. We will split the boundary into three types of pieces.
For strongly convex boundary pieces, we use the same method as in Proposition \ref{henkinstrconv}. For pieces in $\mathbf B(0,\epsilon)$, The estimate (\ref{eq:I-strconv}) still holds for $z$ and $\z$ close to $0$.
By the following Proposition \ref{pro:expabs} and Proposition \ref{pro:exprez1},
we will show the estimate is valid near $0$. For the pieces on $\p\mathbf B(0,\epsilon)\cap \W$ where the defining equation is not strictly convex, the same method as in Proposition \ref{pro:expabs} and Proposition \ref{pro:exprez1} can still be applied. Indeed, in the proof of those propositions we only used the estimates in Lemma \ref{conv}. The estimates of $|\re F|$ continue to be valid. We leave this part to the reader.


We follow the same method as in Proposition \ref{henkinstrconv} and
get (\ref{eq:I-strconv}). It comes down to estimating $|\re F|$
near $0$ for each of these two special cases, where \[
F(z,\z)=\rho_{\z_{1}}(z_{1}-\z_{1})+\rho_{\z_{2}}(z_{2}-\z_{2}).\]

\begin{prop}
\label{pro:expabs}Let $\W\cap B(0,\dd)=\set{\rho(z)=\re z_{2}+\exp\br{-1/|z_{1}|^{\alpha}}<0}$.
Then, for sufficiently small $\dd>0$, we have\[
\int_{\sqrt{t_{1}^{2}+t_{2}^{2}}<\dd}\frac{|\ln|\re F(z,\z)||}{|z_{1}-\z_{1}|}dt_{1}dt_{2}<\infty,\]
for $z\in\cl\W\cap B(0,\dd)$, if $\alpha<1$. \end{prop}
\begin{proof}
Let $\phi(t)=\exp\br{-1/t^{\alpha/2}}$. Then we have $\rho(z)=\re z_{2}+\phi(|z_{1}|^{2})$.
Since $\z\in\p\W$, we have $\re\z_{2}=-\phi(|\z_{1}|^{2})$. Hence
we get \begin{align*}
\re F & =\re\phi'(|\z_{1}|^{2})\cl\z_{1}(z_{1}-\z_{1})+\frac{1}{2}\re(z_{2}-\z_{2})\\
 & =\phi'(|\z_{1}|^{2})\re\cl\z_{1}(z_{1}-\z_{1})+\frac{1}{2}\br{\re z_{2}+\phi(|\z_{1}|^{2})}\end{align*}
Since $\re F\le0$ and $\re z_{2}+\phi(|z_{1}|^{2})\le0$, we have
\begin{align*}
|\re F| & =-\re F=\phi'(|\z_{1}|^{2})\bk{|\z_{1}|^{2}-\re\cl\z_{1}z_{1}}-\frac{1}{2}\phi(|\z_{1}|^{2})-\frac{1}{2}\re z_{2}\\
 & \ge\phi'(|\z_{1}|^{2})\bk{|\z_{1}|^{2}-\re\cl\z_{1}z_{1}}-\frac{1}{2}\phi(|\z_{1}|^{2})+\frac{1}{2}\phi(|z_{1}|^{2})\\
 & \ge\frac{1}{2}\bk{\phi'(|\z_{1}|^{2})(|\z_{1}|^{2}-|z_{1}|^{2})-\phi(|\z_{1}|^{2})+\phi(|z_{1}|)^{2}}\end{align*}
Apply Lemma \ref{conv} with $p=|\z_{1}|^{2}$ and $q=|z_{1}|^{2}$.
Then we get\begin{align*}
|\re F| & \ge\frac{1}{2}\phi(|z_{1}|^{2}-|\z_{1}|^{2}),\quad\text{if }|z_{1}|\ge|\z_{1}|\quad\text{and }\\
|\re F| & \ge\frac{1}{2}\phi''\br{\frac{|z_{1}|^{2}+|\z_{1}|^{2}}{2}}\br{\frac{|\z_{1}|^{2}-|z_{1}|^{2}}{2}}^{2}\quad\text{if }|z_{1}|<|\z_{1}|.\end{align*}
Hence we have\begin{multline*}
\int_{|t|<\dd}\frac{|\ln|\re F(z,\z)||}{|z_{1}-\z_{1}|}dt_{1}dt_{2}\lesssim\int_{|t|<\dd,|z_{1}|\ge|\z_{1}|}\frac{|\ln\phi(|z_{1}|^{2}-|\z_{1}|^{2})|}{|z_{1}-\z_{1}|}dt_{1}dt_{2}\\
+\int_{|t|<\dd,|z_{1}|<|\z_{1}|}\frac{|\ln\br{\phi''((|z_{1}|^{2}+|\z_{1}|^{2})/2)(|\z_{1}|^{2}-|z_{1}|^{2})^{2}/4}|}{|z_{1}-\z_{1}|}dt_{1}dt_{2}\\
\lesssim\int_{|t|<\dd}\frac{|\ln\phi(|z_{1}|^{2}-|\z_{1}|^{2})|}{|z_{1}-\z_{1}|}dt_{1}dt_{2}+\int_{|t|<\dd}\frac{|\ln\br{\phi''((|z_{1}|^{2}+|\z_{1}|^{2})/2)(|\z_{1}|^{2}-|z_{1}|^{2})^{2}/4}|}{|z_{1}-\z_{1}|}dt_{1}dt_{2}\\
=I+II.\end{multline*}
We shall show that $I<\infty$ and $II<\infty.$ Let us first consider
$I$: \begin{multline}
I\lesssim\int\frac{|\ln\phi(|z_{1}|^{2}-|\z_{1}|^{2})|}{|z_{1}-\z_{1}|}dt_{1}dt_{2}
\lesssim\int\frac{1}{\Abs{|z_{1}|^{2}-|\z_{1}|^{2}}^{\alpha/2}|z_{1}-\z_{1}|}dt_{1}dt_{2}\label{eq:eabs}\end{multline}
Let us use the polar coordinate centered at $z_{1}$ and let $\z_{1}=z_{1}+re^{i\theta}$.
Then we get\begin{align*}
I\lesssim&\int_{0}^{2\pi}\int_{0}^{1}[\frac{1}{r(r+(z_1 e^{-i\theta}+\bar z_1 e^{i\theta}))}]^{\alpha/2}dr d\theta\\
\lesssim& \int_{0}^{2\pi}\int_{0}^{1} [\frac{1}{r^2}+\frac{1}{(r+(z_1 e^{-i\theta}+\bar z_1 e^{i\theta}))^2}]^{\alpha/2}dr d\theta\\
\lesssim&  \int_{0}^{2\pi}\int_{0}^{1}(\frac{1}{r})^{\alpha}+ \big[\frac{1}{r+(z_1 e^{-i\theta}+\bar z_1 e^{i\theta})}\big]^{\alpha}dr d\theta  <\infty,\end{align*}
if $\alpha<1$.

Now let us consider $II$. We have \begin{align}
\phi'(t) & =\exp\br{-1/t^{\alpha/2}}\frac{\alpha}{2}\frac{1}{t^{1+\alpha/2}};\nonumber \\
\phi''(t) & =\exp\br{-1/t^{\alpha/2}}\frac{1}{t^{2+\alpha}}\frac{\alpha}{2}\bk{\frac{\alpha}{2}-\br{1+\frac{\alpha}{2}}t^{\alpha/2}}\nonumber \\
 & \ge C(\alpha)\exp\br{-1/t^{\alpha/2}}\frac{1}{t^{2+\alpha}},\label{eq:phi''}\end{align}
if we choose $\dd>0$ sufficiently small such that that $\phi''(t)\ge0$
for all $t<\dd^{2}$. Therefore we get \begin{align*}
II & \lesssim\int_{|t|<\dd}\br{\frac{1}{(|z_{1}|^{2}+|\z_{1}|^{2})^{\alpha/2}}+\Big|\ln||z_{1}|^{2}-|\z_{1}|^{2}|\Big|}\frac{1}{|z_{1}-\z_{1}|}dt_{1}dt_{2}\\
 & \lesssim\int_{|t|<\dd}\br{\frac{1}{|z_{1}-\z_{1}|^{\alpha}}}\frac{1}{|z_{1}-\z_{1}|}dt_{1}dt_{2}<\infty,\end{align*}
if $\alpha<1$.
\end{proof}

\begin{prop}
\label{pro:exprez1}Let $\W\cap B(0,\dd)=\set{\rho(z)=\re z_{2}+\exp\br{-1/|\re z_{1}|^{\alpha}}<0}$.
Then we have\[
\int_{\sqrt{t_{1}^{2}+t_{2}^{2}}<\dd}\frac{|\ln|\re F(z,\z)||}{|z_{1}-\z_{1}|}dt_{1}dt_{2}<\infty,\]
for $z\in\cl\W\cap B(0,\dd)$, if $\alpha<1$. \end{prop}
\begin{proof}
Let $\phi(t)=\exp\br{-1/t^{\alpha/2}}$. Then we have $\rho(z)=\re z_{2}+\phi(|\re z_{1}|^{2})$.
Since $\z\in\p\W$, we have $-\re\z_{2}=\phi(|\re\z_{1}|^{2})$. Hence
we get \begin{align*}
2\re F & =\phi'(t_{1}^{2})2t_{1}(x_{1}-t_{1})+(\re z_{2}-\re\z_{2})\\
 & =\phi'(t_{1}^{2})2t_{1}(x_{1}-t_{1})+\br{\re z_{2}+\phi(t_{1}^{2})}\end{align*}
Since $\re F\le0$ and $\re z_{2}+\phi(|z_{1}|^{2})\le0$, we have
\begin{align*}
|2\re F| & =-2\re F=2\phi'(t_{1}^{2})(t_{1}^{2}-t_{1}x_{1})-(\re z_{2}+\phi(t_{1}^{2}))\\
  & \ge\phi'(t_{1}^{2})\br{t_{1}^{2}-x_{1}^{2}}-\phi(t_{1}^{2})+\phi(x_{1}^{2}).\end{align*}
Apply Lemma \ref{conv} with $p=t_{1}^{2}$ and $q=x_{1}^{2}$. Then
we get \begin{align*}
|\re F| & \ge\frac{1}{2}\phi(x_{1}^{2}-t_{1}^{2})\quad\text{if }x_{1}^{2}\ge t_{1}^{2}\quad\text{and }\\
|\re F| & \ge\frac{1}{2}\phi''\br{\frac{x_{1}^{2}+t_{1}^{2}}{2}}\br{\frac{t_{1}^{2}-x_{1}^{2}}{2}}^{2}\quad\text{if }x_{1}^{2}<t_{1}^{2}.\end{align*}
Hence we have\begin{multline*}
\int_{|t|<\dd}\frac{|\ln|\re F(z,\z)||}{|z_{1}-\z_{1}|}dt_{1}dt_{2}\\
\lesssim\int_{|t|<\dd,|x_{1}|\ge|t_{1}|}\frac{|\ln\phi(x_{1}^{2}-t_{1}^{2})|}{|z_{1}-\z_{1}|}dt_{1}dt_{2}+\int_{|t|<\dd,|x_{1}|<|t_{1}|}\frac{|\ln\br{\phi''((x_{1}^{2}+t_{1}^{2})/2)(t_{1}^{2}-x_{1}^{2})^{2}/4}|}{|z_{1}-\z_{1}|}dt_{1}dt_{2}\\
\lesssim\int_{|t|<\dd}\frac{|\ln\phi(|x_{1}^{2}-t_{1}^{2}|)|}{|z_{1}-\z_{1}|}dt_{1}dt_{2}+\int_{|t|<\dd}\frac{|\ln\br{\phi''((x_{1}^{2}+t_{1}^{2})/2)(t_{1}^{2}-x_{1}^{2})^{2}/4}|}{|z_{1}-\z_{1}|}dt_{1}dt_{2}=I+II.\end{multline*}
We shall show that $I,II<\infty$. First let us consider $I$. If
we assume $\alpha<1$, then we can find $\e>0$ such that $\alpha+2\e<1$
and we have\begin{alignat*}{1}
I & \lesssim\int\frac{1}{|x_{1}^{2}-t_{1}^{2}|^{\alpha/2}|z_{1}-\z_{1}|}dt_{1}dt_{2}\\
 & =\int\frac{1}{|x_{1}+t_{1}|^{\alpha/2}|x_{1}-t_{1}|^{\alpha/2}|x_{1}-t_{1}|^{\e}|x_{2}-t_{2}|^{1-\e}}dt_{1}dt_{2}\\
 & \lesssim\br{\int_{|t_{1}|<1}\frac{1}{|x_{1}+t_{1}|^{\alpha}}dt_{1}}^{1/2}\br{\int_{|t_{1}|<1}\frac{1}{|x_{1}-t_{1}|^{\alpha+2\e}}dt_{1}}^{1/2}<\infty.\end{alignat*}
Let us consider $II$. From (\ref{eq:phi''}), we get\begin{align*}
II & \lesssim\int_{|t|<\dd}\br{\frac{1}{(x_{1}^{2}+t_{1}^{2})^{\alpha/2}}+|\ln|x_{1}^{2}-t_{1}^{2}||}\frac{1}{|z_{1}-\z_{1}|}dt_{1}dt_{2}\\
 & \lesssim\int_{|t|<\dd}\br{\frac{1}{|x_{1}-t_{1}|^{\alpha}}}\frac{1}{|x_{1}-t_{1}|^{\e}|x_{2}-t_{2}|^{1-\e}}dt_{1}dt_{2}<\infty,\end{align*}
if we choose $\e>0$ such that $\alpha+\e<1$. \end{proof}
The proof of Theorem 5 is thus complete. \end{proof}

\section{rounding off a ball cut by a cylinder}

Let's recall Example 1:

Let $\chi:\R^{+}\cup\set0\To\R^+$ be a smooth function such that $\chi''(t)\ge0$
everywhere, $\chi''(t)>0$ for all $t\in(1,1+a)$, and \[
\chi(t)=\begin{cases}
1, & t\in[0,1]\\
1+\exp\br{-\frac{1}{(t-1)^{\frac{\alpha}{2}}}}, & t\in(1,1+\e)\\
t-\eta, & t\ge1+a\end{cases},\]
 where $a>\e>0$ are small numbers, $0<\eta<a$. Let us define a domain
$\W\subset\C^{2}$ as follows: \[
\W=\set{\rho(z_{1},z_{2})=\chi(\abs{z_{1}}^{2})+\abs{z_{2}}^{2}<4}.\]
 \begin{prop} If $\alpha<1$, then the solution $u$ to the $\bar{\p}$ equation for the
domain in Example 1  has supnorm estimates $\|u\|_{\infty}\lesssim \|f\|_{\infty}$. \end{prop}

\begin{proof} Note that \begin{multline}
\p\W=\set{\abs{z_{1}}\le1,\abs{z_{2}}^{2}=3}\\
\frac{}{}\cup\set{1<\abs{z_{1}}^{2}<1+a,\,\chi(\abs{z_{1}}^{2})+\abs{z_{2}}^{2}=4}\\
\cup\set{\abs{z_{1}}^{2}>1+a,\,\abs{z_{1}}^{2}+\abs{z_{2}}^{2}=4+\eta}=P_{1}\cup P_{2}\cup P_{3}.\end{multline}
 Solve $\dbar u=f=f_{1}d\cl z_{1}+f_{2}d\cl z_{2}$ on $\W$ using
Henkin integral:

\begin{multline}
4\pi^{2}u(z)=\int_{\p\W}\frac{\rho_{\z_{1}}(\cl\z_{2}-\cl z_{2})-\rho_{\z_{2}}(\cl\z_{1}-\cl z_{1})}{F(z,\z)\abs{\z-z}^{2}}f\wedge d\z_{1}\wedge d\z_{2}\\
+\int_{\W}\frac{f_{1}(\cl\z_{1}-\cl z_{1})+f_{2}(\cl\z_{2}-\cl z_{2})}{\abs{\z-z}^{4}}d\cl\z_{1}\wedge d\cl\z_{2}\wedge d\z_{1}\wedge d\z_{2}\\
=\int_{P_{1}}\cdots+\int_{P_{2}}\cdots+\int_{P_{3}}\cdots+\int_{\W}\cdots=u_{1}+u_{2}+u_{3}+u_{4},\end{multline}
 where \[
F(z,\z)=\rho_{\z_{1}}(z_{1}-\z_{2})+\rho_{\z_{2}}(z_{2}-\z_{2}).\]

We shall show that \[
\norm{u_{j}}_{\infty}\lesssim\norm f_{\infty},\quad j=1,2,3,4.\]
 The proof of  $\norm{u_{4}}_{\infty}\lesssim\norm f_{\infty}$ has been included in the literature. The estimates on $u_2$ and $u_3$ are also omitted here, as they are exactly the same as in Proposition \ref{henkin2} and Proposition \ref{henkinstrconv}.

For the estimate of $\norm{u_{2}}_{\infty}$, we first compute \begin{align*}
 & F(z,\z)=\frac{\p\rho}{\p\z_{1}}(\z)(z_{1}-\z_{1})+\frac{\p\rho}{\p\z_{2}}(\z)(z_{2}-\z_{2})\\
 & =e^{-\frac{1}{(\abs{\z_{1}}^{2}-1)^{\frac{\alpha}{2}}}}\frac{\alpha\bar{\z}_{1}(z_{1}-\z_{1})}{2(\abs{\z_{1}}^{2}-1)^{\frac{\alpha}{2}+1}}+\bar{\z}_{2}(z_{2}-\z_{2}).\end{align*}



Notice on $P_2$, it's strictly convex except at those points $(\z_1, \z_2)$ so that $(|\z_{1}|,|\z_{2}|)=(1,\sqrt{3})$,
The integrand in $u_{2}$ becomes most singular when $(z_{1},z_{2})$
is near those boundary points. From now on we assume
$(|z_{1}|,|z_{2}|)\in \bar\W, (|\z_{1}|,|\z_{2}|)\in \p\W$ are both close to $(1,\sqrt{3})$. We can also assume that $f=f_{1}d\bar{z}_{1}$. Indeed, since $\rho=0$, $
\rho_{\z_{1}}d\z_{1}+\rho_{\bar{\z}_{1}}d\bar{\z}_{1}+\rho_{\z_{2}}d\z_{2}+\rho_{\bar{\z}_{2}}d\bar{\z}_{2}=0$ {on} $\ \p\Omega.$
 Therefore $
d\bar{\z}_{2}=d_{1}d\z_{1}+d_{2}d\bar{\z}_{1}+d_{3}d\z_{2},$
 where $d_{1},d_{2}=o(1)$. Then $
d\bar{\z}_{2}\wedge d\z_{1}\wedge d\z_{2}=o(1)d\bar{\z}_{1}\wedge d\z_{1}\wedge d\z_{2}.$

\begin{align*}
|A| & \le\|f\|_{\infty}\int_{1<\abs{\z_{1}}<1+\epsilon,\chi(|\z_{1}|^{2})+|\z_{2}|^{2}=4}\frac{1}{|F||z_{1}-\z_{1}|}d\bar{\z}_{1}\wedge d\z_{1}\wedge d\z_{2}\\
 & \ \ +\|f\|_{\infty}\int_{1+\epsilon<\abs{\z_{1}}<1+a,\chi(|\z_{1}|^{2})+|\z_{2}|^{2}=4}\frac{1}{|F||z_{1}-\z_{1}|}d\bar{\z}_{1}\wedge d\z_{1}\wedge d\z_{2}:=I+II.\\
 I& \le\|f\|_{\infty}\int_{1<\abs{\z_{1}}<1+\epsilon,\chi(|\z_{1}|^{2})+|\z_{2}|^{2}=4}\frac{1}{|F||z_{1}-\z_{1}|}d\bar{\z}_{1}\wedge d\z_{1}\wedge d\z_{2}\\
 & =\|f\|_{\infty}\int_{1<\abs{\z_{1}}<1+\epsilon,\chi(|\z_{1}|^{2})+|\z_{2}|^{2}=4}\frac{1}{|z_{1}-\z_{1}|(|\re F|+|\im F|)}d\z_{2}\wedge d\bar{\z}_{1}\wedge d\z_{1}.\\
 II & \lesssim \|f\|_{\infty}.
 \end{align*}

Since $\frac{\p\im F}{\p\z_{2}}\approx1$ as $\z_{2}\approx\sqrt{3}$,
by taking change of coordinates, we can assume the last integration
is actually over $d\im F\wedge d\bar{\z}_{1}\wedge d\z_{1}$. Therefore

\begin{align*}
I & \le\|f\|_{\infty}\int_{1<\abs{\z_{1}}<1+\epsilon}\frac{\ln(|\re F|)}{|z_{1}-\z_{1}|}d\bar{\z}_{1}\wedge d\z_{1}.\end{align*}

Here, \begin{align}\label{mm}
\re F & =e^{-\frac{1}{(\abs{\z_{1}}^{2}-1)^{\frac{\alpha}{2}}}}\frac{\alpha\re(\bar{\z}_{1}(z_{1}-\z_{1}))}{2(\abs{\z_{1}}^{2}-1)^{\frac{\alpha}{2}+1}}+\re(\bar{\z}_{2}z_{2})-3+e^{-\frac{1}{(\abs{\z_{1}}^{2}-1)^{\frac{\alpha}{2}}}}.
\end{align}
We need the following lemma:

\begin{lem}
Let $\re F$ be given by (\ref{mm}) and let $\phi(t)=e^{-\frac{1}{t^{\frac{\alpha}{2}}}}$ for $t>0$, $\phi(0)=0$. Denote $y=\re\bar{\z}_{1}z_{1}-1$
and $x=|\z_{1}|^{2}-1$. Then \begin{equation}\label{secl}|\re F|\ge \begin{cases} \phi(y-x), & \text{if}\ \  y\ge x\ge 0 \\
\phi''\br{\frac{x-y}{2}}\br{\frac{x-y}{2}}^{2}, & \text{if}\ \ 0\le y\le x\\
\phi''\br{\frac{x}{2}}\br{\frac{x}{2}}^{2}, & \text{if}\ \  y\le0\end{cases}.\end{equation}
\end{lem} \begin{proof}
Since $\z\in P_2,z\in\Omega$, we have \begin{align*}
 & \phi(|\z_{1}|^{2}-1)+|\z_{2}|^{2}-3=0,\\
 & \phi(\big||z_{1}|^{2}-1\big|)+|z_{2}|^{2}-3<0.
 \end{align*}
 Notice $\re\bar{\z}_{1}z_{1}\le \frac{1}{2}|\z_{1}|^{2}+\frac{1}{2}|z_{1}|^{2}$,
$\re\bar{\z}_{2}z_{2}\le \frac{1}{2}|\z_{2}|^{2}+\frac{1}{2}|z_{2}|^{2}$.
Moreover, $\phi$ is increasing and convex when $t$ is close to
$0^+$. So \begin{align*}
\phi(|y|) & \le \phi(\frac{1}{2}(|\z_{1}|^{2}-1)+\frac{1}{2}\big||z_{1}|^{2}-1)\big|)\le\frac{1}{2}\phi(|\z_{1}|^{2}-1)+\frac{1}{2}\phi(\big||z_{1}|^{2}-1\big|)\\
 & \le \frac{1}{2}(3-|\z_{2}|^{2})+\frac{1}{2}(3-|z_{2}|^{2})\\
 & =3-(\frac{1}{2}|\z_{2}|^{2}+\frac{1}{2}|z_{2}|^{2})\\
 & \le 3-\re\bar{\z}_{2}z_{2}.\end{align*}
 Therefore if $y\ge0$, \begin{align*}
\re F & =\phi'(x)(y-x)+\re(\z_{2}z_{2})-3+\phi(x)\\
 & \le \phi'(x)(y-x)-\phi(y)+\phi(x).\end{align*}
 Applying Lemma \ref{conv} to the above, we have the estimates in (\ref{secl}).

 If $y\le 0$,
 \begin{align*}
\re F & =\phi'(x)(y-x)+\re(\z_{2}z_{2})-3+\phi(x)\\
 & = \phi'(x)(0-x)-\phi(0)+\phi(x)\\
 & \le  -\phi''\br{\frac{x}{2}}\br{\frac{x}{2}}^{2} .\end{align*} \end{proof}

We then
have \begin{align*}
I\lesssim &\|f\|_{\infty}\{\int_{1<\abs{\z_{1}}<1+\epsilon, y\ge x}\frac{1}{|\re (\bar\z_1z_1-|\z_1|^2)|^{\frac{\alpha}{2}}|z_1-\z_1|} d\bar\z_1\wedge d\z_1 \\
&\int_{1<\abs{\z_{1}}<1+\epsilon, 0\le y \le x}\frac{|\ln(|\z_1|^2-\re(\bar\z_1z_1))|}{(|\z_1|^2-\re (\bar\z_1z_1))^{\frac{\alpha}{2}}|z_1-\z_1|} d\bar\z_1\wedge d\z_1\\
& +\int_{1<\abs{\z_{1}}<1+\epsilon, y\le 0}\frac{\big|\ln(|\z_1|^2-1)|\big|}{ (|\z_1|^2-1)^{\frac{\alpha}{2}}|z_1-\z_1|} d\bar\z_1\wedge d\z_1\}\\
:=&\|f\|_{\infty}(A_1+A_2+A_3).\end{align*}
Here,
\begin{align*}
  A_1&\le \{\int_{1<\abs{\z_{1}}<1+\epsilon, y\ge x}\big[\re (\frac{1}{\bar\z_1}\frac{1}{z_1-\z_1})\big]^{\frac{\alpha}{2}}|\frac{1}{z_1-\z_1}| d\bar\z_1\wedge d\z_1\\
  &\le \{\int_{1<\abs{\z_{1}}<1+\epsilon}\big[\frac{1}{|\bar\z_1|}\frac{1}{|z_1-\z_1|}\big]^{\frac{\alpha}{2}} |\frac{1}{z_1-\z_1}|d\bar\z_1\wedge d\z_1\\
  &\lesssim \{\int_{1<\abs{\z_{1}}<1+\epsilon}|\frac{1}{z_1-\z_1}|^{1+\frac{\alpha}{2}} d\bar\z_1\wedge d\z_1<\infty\\
  A_2&\le \int_{1<\abs{\z_{1}}<1+\epsilon}\frac{|\ln|\z_1-z_1||}{|\z_1-z_1|^{1+\frac{\alpha}{2}}} d\bar\z_1\wedge d\z_1<\infty,\\
\end{align*}
 when $\alpha<1$.

For $A_3$, we need to use the inequality $ a^{p}+b^q\gtrsim ab$ for $ a, b\ge 0$ and $1/p+1/q=1, p, q>1$. Now since $\alpha<1$, we can choose  $p_0>2$ so that $\frac{p_0\alpha}{2}<1$ and $q_0$ satisfy $1/p_0+1/q_0=1$. Then $q_0<2$.
\begin{align*}
  A_3&\le  \int_{1<\abs{\z_{1}}<1+\epsilon}\frac{\big|\ln(|\z_1|^2-1)|\big|^{p_0}}{ (|\z_1|^2-1)^{\frac{p_0\alpha}{2}}}d\bar\z_1\wedge d\z_1+\int_{1<\abs{\z_{1}}<1+\epsilon}\frac{1}{|z_1-\z_1|^{q_0}} d\bar\z_1\wedge d\z_1\\
  &\le \int_{1<\abs{\z_{1}}<1+\epsilon}\frac{\big|\ln(|\z_1|^2-1)|\big|^{p_0}}{ (|\z_1|^2-1)^{\frac{p_0\alpha}{2}}}d\bar\z_1\wedge d\z_1+const.
  \end{align*}

In the first integral, take the polar coordinates $\z=re^{i\theta}$. Then
\begin{align*}
  A_3&\lesssim \int_{1<r<1+\epsilon}\frac{|\ln(r^2-1)|^{p_0}}{(r^2-1)^{\frac{p_0\alpha}{2}}}dr+const\\
   & \lesssim \int_{1<r<1+\epsilon} \frac{|\ln(r-1)|^{p_0}}{(r-1)^{\frac{p_0\alpha}{2}}}dr+const< \infty ,
\end{align*}
by our choices of $p_0, q_0$. \end{proof}

\medskip{}

One can similarly consider the domain obtained by rounding off the
corners of a bidisc. See the following example:

\begin{example} Let \[
\chi(t)=\bcases1-a,&t\le1-a\\
k\exp\br{-\Frac{1}{(t^{2}-(1-a)^{2})^{\frac{\alpha}{2}}}}+1-a,&t>1-a,\ecases\]
 where $a>0$ is a small constant such that $\chi(t)$ is convex on
$[0,1]$ and $k$ is a constant chosen such that $\chi(1)=1$, i.e.,
\[
k=a\cdot\exp\br{\frac{1}{(2a-a^{2})^{\frac{\alpha}{2}}}}.\]

Let \[
\W=\set{\rho(z)=\chi(\abs{z_{1}})+\chi(\abs{z_{2}})-2+a<0}\subset\C^{2}.\]

\end{example}

As we can see, the boundary of $\W$ consists of the following sets:
\begin{multline}
\p\W=\set{\abs{z_{1}}\le1-a,\,\abs{z_{2}}=1}\cup\set{\abs{z_{1}}=1,\,\abs{z_{2}}\le1-a}\\
\cup\set{\abs{z_{1}},\abs{z_{2}}>1-a,\,\rho(z_{1},z_{2})=0}=P_{1}\cup P_{2}\cup P_{3}.\end{multline}

Integrals on flat pieces $P_{1}$ and $P_{2}$ fall into the Henkin's
bidisc situation. The supernorm estimates of the integral over rounding
off piece $P_{3}$ can be carried out the same way as we discussed in Example
1 when $\alpha<1$.

\bigskip

\noindent John Erik Forn\ae ss\\
Mathematics Department\\
The University of Michigan\\
East Hall, Ann Arbor, MI 48109\\
USA\\
fornaess@umich.edu\\

\noindent Lina Lee\\
Mathematics Department\\
The University of Michigan\\
East Hall, Ann Arbor, MI 48109\\
USA\\
linalee@umich.edu\\

\noindent Yuan Zhang\\
Mathematics Department\\
University of California, San Diego\\
9500 Gilman Drive, La Jolla, CA, 92093\\
USA\\
yuz009@math.ucsd.edu\\

\end{document}